\documentclass[12pt]{amsart}

\usepackage{amssymb,amsmath}
\usepackage{amssymb,latexsym}
\usepackage{amsfonts}
\usepackage{amssymb}
\usepackage{lscape}
\usepackage{stmaryrd}
\usepackage{longtable}
\usepackage{hyperref}
\usepackage{amsaddr}
\usepackage{amsmath, geometry, amssymb} % see geometry.pdf on how to lay out the page. There's lots.
\numberwithin{equation}{section}
\newtheorem{thm}{Theorem}[section]

\numberwithin{table}{section}

\newcommand{\sqr}[2]{{\vcenter{\vbox{\hrule height#2pt
 \hbox{\vrule width#2pt height#1pt \kern#1pt
 \vrule width#2pt}\hrule height#2pt}}}}

\newcommand{\beq}{\begin{equation}}
\newcommand{\eeq}{\end{equation}}
\newcommand{\beqar}{\begin{eqnarray}}
\newcommand{\eeqar}{\end{eqnarray}}
\def\beqars{\begin{eqnarray*}}
\def\eeqars{\end{eqnarray*}}

\newcommand{\pmd}{\hspace{-3mm} \pmod}

\newcommand{\smod}[1]{\hspace{-1mm} \pmod{#1}}

\newcommand{\qu}[2]{\Bigl({\frac{#1}{#2}}\Bigr) }
\newcommand{\dqu}[2]{\ds{\qu{#1}{#2}}}

\def \ds{\displaystyle}

\newcommand{\nn}{\mathbb{N}}
\newcommand{\zz}{\mathbb{Z}}
\newcommand{\qq}{\mathbb{Q}}
\newcommand{\cc}{\mathbb{C}}

\newcommand{\hh}{\mathbb{H}}

\allowdisplaybreaks

\begin{document} 

\markboth{Zafer Selcuk Aygin}
{Representations by sextenary quadratic forms}

\title{Representations by sextenary quadratic forms with coefficients $1,2,3$ and $6$ and on newforms in $S_{3} (\Gamma_0 (24), \chi )$ }

\author{Zafer Selcuk Aygin}
\address{Division of Mathematical Sciences, School of Physical and Mathematical Sciences, Nanyang Technological University, 21 Nanyang Link, Singapore 637371, Singapore} 
\email{selcukaygin@ntu.edu.sg}

\maketitle

\begin{abstract}
We use theory of modular forms to give formulas for $N(1^{l_1},2^{l_2},3^{l_3},6^{l_6};n)$ for all $l_1,l_2,l_3,l_6 \in \nn_0$, with $l_1+l_2+l_3+l_6=6$. We also apply our results to write newforms in $S_{3} (\Gamma_0 (24), \chi )$ in terms of eta quotients.
\end{abstract}

\keywords{Dedekind eta function; eta quotients; eta products; theta functions; Eisenstein series; Eisenstein forms; modular forms; cusp forms; Fourier coefficients; Fourier series.}
\noindent
Mathematics Subject Classification: 11F11, 11F20, 11F27, 11E20, 11E25, 11F30

%%% Section 1

\section{Introduction} \label{sec:1}
Let $\nn$, $\nn_0$, $\zz$, $\cc$ and $\hh$ denote the sets of positive integers, non-negative integers, integers, complex numbers and the upper half plane, respectively. We use the notation $q=e(z):=e^{2\pi i z}$ with $z\in \hh$, and so $|q| < 1$. Let $k , N \in \nn$ and $\Gamma_0(N)$ be the modular subgroup defined by
\beqars
\Gamma_0(N) = \left\{ \left(
\begin{array}{cc}
a & b \\
c & d
\end{array}
\right) \mid a,b,c,d\in \zz ,~ ad-bc = 1,~c \equiv 0 \smod {N}
\right\} .
\eeqars 
We write $M_k(\Gamma_0(N), \chi)$ to denote the space of modular forms of weight $k$ for $\Gamma_0(N)$ with multiplier $\chi$, and 
$E_k (\Gamma_0(N), \chi)$ and $S_k(\Gamma_0(N), \chi)$ to denote the subspaces of Eisenstein forms and cusp forms 
of $M_k(\Gamma_0(N), \chi)$, respectively. 
It is known (see \cite[p. 83]{stein}, \cite[Theorem 2.1.7]{miyake1}) that
\beqar
M_k (\Gamma_0(N),\chi) = E_k (\Gamma_0(N),\chi) \oplus S_k(\Gamma_0(N),\chi). \label{1_1}
\eeqar
Let $\chi$ and $\psi$ be primitive characters. For $n\in \nn$ we define $\ds \sigma_{(k,\chi,\psi )}(n)$ by
\beqar
\sigma_{(k, \chi,\psi )}(n) =\sum_{1 \leq d\mid n}\chi(d)\psi(n/d)d^k. \label{3_1}
\eeqar 
If $n \not\in \nn$ we set $\sigma_{(k,\chi,\psi )}(n)=0$. For each quadratic discriminant $t$, we put $\chi_{_t}(n)=\dqu{t}{n}$, where $\dqu{t}{n}$ is Kronecker symbol defined by \cite[p. 296]{vaughan}.

Suppose $k, N \in \nn$. Let $\chi$ and $\psi$ be primitive Dirichlet characters such that $\chi(-1)\psi(-1)=-1$ and with conductors $L,R \in \nn$, respectively. The weight $3$ Eisenstein series are defined by
\beqar
&& E_{3,\chi, \psi}(z)=c_0+\sum_{n \geq 1} \sigma_{(2, \chi,\psi )}(n) q^n \label{3_2}
\eeqar
where 
\beqars
c_0=\left\{\begin{array}{rl}
-B_{3,\chi}/{6} & \mbox{ if $R=1$} \\
0 & \mbox{ if \mbox{ $R>1$},} 
\end{array} \right.
\eeqars 
and the generalized Bernoulli numbers $B_{3,\chi}$ attached to $\chi$ are defined by the following equation:
\beqars
&& B_{3,\chi}= 6 [x^3] \sum_{a=1}^L\frac{\chi(a)x e^{ax}}{e^{Lx}-1},
\eeqars 
from which we compute
\begin{align*}
B_{3,\chi_{_{-3}}}=2/3, ~ B_{3,\chi_{_{-4}}}=3/2,~ B_{3,\chi_{_{-8}}}=9, ~B_{3,\chi_{_{-24}}}=138.
\end{align*}
We use Eisenstein series in Section \ref{sec:8_1} to give bases for $E_3 (\Gamma_0(24), \chi)$, see \cite[Theorem 5.9]{stein}. 

The Dedekind eta function $\eta (z)$ is the holomorphic function defined on the upper half plane $\hh$ 
by the product formula
\beqars
\eta (z) = q^{1/24} \prod_{n=1}^{\infty} (1-q^{n}).
\eeqars
A product of the form
\beqar \label{1_40}
f(z) = \prod_{1\leq \delta \mid N} \eta^{r_{\delta}} ( \delta z) ,
\eeqar 
where $r_{\delta} \in \zz$, not all zero, is called an eta quotient. We use eta quotients to express the generating functions for number of representations by quadratic forms and to give bases for $S_k (\Gamma_0(N), \chi)$. Let $\delta_1=1,~\delta_2,\ldots,~\delta_{d-1},~\delta_{d}=N$ be the divisors of $N$, in ascending order. For convenience, we use the following shorthand notation
\beqars
\eta_N[r_{\delta_1},~r_{\delta_2},\ldots,~r_{\delta_{d-1}},~r_{\delta_{d}}](z)= \eta^{r_{\delta_1}} ( \delta_1 z) \eta^{r_{\delta_2}} ( \delta_2 z) \ldots \eta^{r_{\delta_{d-1}}} ( \delta_{d-1} z) \eta^{r_{\delta_{d}}} ( \delta_{d} z).
\eeqars
The order of zeros of an eta quotient given by (\ref{1_40}) at the cusp $a/c \in \qq$ is
\beqar
v_{a/c}(f)=\frac{N}{24\gcd(c^2,N)}\sum_{1 \leq \delta|N}\frac{\gcd(\delta,c)^2 \cdot r_\delta }{\delta} , \label{etaqorder}
\eeqar
see \cite[Proposition 3.2.8, p. 34]{Ligozat}.

Let $m \in \nn$, $r_i \in \nn_0$, and $a_i \in \nn$ for all $1 \leq i \leq m$. Let 
\begin{align*}
N(a_1^{r_1},a_2^{r_2},\ldots,a_{m}^{r_m}; n)
\end{align*}
denote the number of representations of $n$ by the quadratic form
\begin{align}
\sum_{i=1}^{m} \sum_{j=1}^{r_i} a_{i}x_j^2. \label{eq:1}
\end{align}
Ramanujan's theta function $\varphi (z)$ is defined by
\beqars
\varphi(z) = \sum_{n=-\infty}^\infty q^{ n^2 },
\eeqars
thus the generating function of number of representations of $n$ by the quadratic form (\ref{eq:1}) is given by
\beqars
\sum_{n=0}^{\infty} N(a_1^{r_1},a_2^{r_2},\ldots,a_{m}^{r_m}; n) q^n= \prod_{i=1}^m \varphi^{r_i}(a_i z) . 
\eeqars

On the other hand by Jacobi's triple product identity \cite[p. 10]{Berndt} we have
\beqar \label{1_50}
\varphi(z)=\frac{\eta^5(2z)}{\eta^2(z) \eta^2(4z)}. 
\eeqar
That is, we can rewrite the generating function in terms of eta quotients as follows:
\beqars
\sum_{n=0}^{\infty} N(a_1^{r_1},a_2^{r_2},\ldots,a_{m}^{r_m}; n) q^n= \prod_{i=1}^m \left( \frac{\eta^5(2a_iz)}{\eta^2(a_iz) \eta^2(4a_iz)} \right)^{r_i} . 
\eeqars
Finding the formulas for number of representations of a number by quadratic forms is an interesting subject in number theory. See \cite{grosswald} for a classical history of this research. For contemporary accounts of the subject see, \cite{alacakesici, alacaw, chanchua, cooper2, hamieh, kokluce-1, milne, onorepr}. Recently in \cite{aaw_1,aaw_2,aaw_3,aaw_4, berkovicyesilyurt, xiayao_1} representations by sextenary quadratic forms was studied. Formulas for all of the diagonal forms with coefficients $1$, $2$ and $4$, and some of the diagonal forms with coefficients $1$ and $3$ are given in these works. In this paper we use theory of modular forms to give formulas for all diagonal sextenary quadratic forms with coefficients $1$, $2$, $3$ and $6$, i.e. we give formulas for
\begin{align}
N(1^{l_1},2^{l_2},3^{l_3},6^{l_6};n) \mbox{ for all $l_1,l_2,l_3,l_6 \in \nn_0$, with $\ds \sum_{i\mid 6} l_i=6$}. \label{9_3}
\end{align}
The formulas for all the $84$ sextenary quadratic forms are given in Tables \ref{table:9_1}--\ref{table:9_4}. Among them only $9$ were previously known in the papers mentioned above, they all agree with our results, due to different choices of eta quotients to express cusp parts of the formulas. We chose the bases for cusp form spaces in a way that each eta quotient chosen to be in the basis have different orders of zeros at infinity. We also use these bases of cusp form spaces to write newforms in terms of eta quotients.

This paper is organized as follows. In Section \ref{sec:2}, we determine the modular spaces of the generating functions of (\ref{9_3}). In Section \ref{sec:8_1} we construct the bases for these modular form spaces. In Section \ref{sec:9_1}, we use those bases obtained in Section \ref{sec:8_1}, to give formulas for (\ref{9_3}). In the last section we use bases of cusp form spaces obtained in Section \ref{sec:8_1} to write some newforms in terms of eta quotients. 

\section{Preliminary results}\label{sec:2}
In this section we use Theorem \ref{th:4_30} which is referred to as Ligozat's Criteria, see \cite[Theorem 5.7, p. 99]{Kilford}, \cite{Ligozat}, and \cite[Proposition 1, p. 284]{Lovejoy}, to determine the modular form spaces of the generating functions of (\ref{9_3}).

\begin{thm}{\label{th:4_30}}
Let $f(z)$ be an eta quotient given by {\em (\ref{1_40})} which satisfies the following conditions\\

{\rm (L1)~} $\ds \sum_{ 1\leq \delta \mid N} \delta \cdot r_{\delta} \equiv 0 \smod {24}$,\\

{\rm (L2)~} $\ds \sum_{ 1 \leq \delta \mid N} \frac{N}{\delta} \cdot r_{\delta} \equiv 0 \smod {24}$, \\

{\rm (L3)~} $\ds v_{1/d}(f(z)) \geq 0 $ for each positive divisor $d$ of $N$, \\

{\rm (L4)~} $\ds k = \frac{1}{2} \sum_{1 \leq \delta \mid N} r_{\delta} $ is a positive integer.\\
\noindent
Then $f(z) \in M_k(\Gamma_0(N),\chi)$ where the character $\chi$ is given by 
\beqar
 \chi (m) = \dqu{(-1)^ks}{m} \label{4_8}\mbox{ with }\ds s= \prod_{p \mid N} \delta^{ \vert \sum_{p \vert \delta } r_{\delta} \vert }. 
\eeqar
Furthermore, if the inequalities in {\rm (L3)} are all strict then $f(z) \in S_k(\Gamma_0(N),\chi)$.
\end{thm} 
We use the multiplicative properties of the Kronecker symbol to give a simpler description of the character given by (\ref{4_8}) for $N= 24$. Let $s'$ be the squarefree part of $s$. Then we have
{\footnotesize
\beqar
 \dqu{(-1)^ks}{m} = \left\{\begin{array}{ll}
\ds \chi_{_{ -s'}}(m), & \mbox{ if $N=24$, $k$ odd, $s'=3$}, \\
\ds \chi_{_{-4s'}}(m), & \mbox{ if $N=24$, $k$ odd, $s'=1,2,6$}.
\end{array} \right. \label{char}
\eeqar }

Next we determine the modular form spaces of generating functions of (\ref{9_3}). 
\begin{thm} \label{th:9_1}
Let $l_d \in \nn_0 $ {\rm ($d \mid 6$)} and $ l_1+l_2+l_3+l_6=6$. Then we have
\beqars
\prod_{d \mid 6} \varphi^{l_d}(d z) \in \left\{\begin{array}{ll}
M_{3}(\Gamma_0(24), \chi_{_{-4}}) & \mbox{ if $ l_1+l_3 \equiv 0 \pmd{2} $} \\
& \mbox{ and $ l_3+l_6 \equiv 0 \pmd{2} $,} \\
M_{3}(\Gamma_0(24),\chi_{_{-3}}) & \mbox{ if $ l_1+l_3 \equiv 0 \pmd{2} $} \\
&\mbox{ and $ l_3+l_6 \equiv 1 \pmd{2} $, } \\
M_{3}(\Gamma_0(24), \chi_{_{-8}}) & \mbox{ if $ l_1+l_3 \equiv 1 \pmd{2} $} \\
& \mbox{ and $ l_3+l_6 \equiv 0 \pmd{2} $,} \\
M_{3}(\Gamma_0(24),\chi_{_{-24}}) & \mbox{ if $ l_1+l_3 \equiv 1 \pmd{2} $}\\
& \mbox{ and $ l_3+l_6 \equiv 1 \pmd{2} $. } 
\end{array} \right.
\eeqars
\end{thm}
\begin{proof}
Let $l_d \in \nn_0 $ {\rm ($d \mid 6$)} and $ l_1+l_2+l_3+l_6=6$, then we have
\begin{align}
\prod_{d \mid 6} \varphi^{l_d}(d z)&=\prod_{d \mid 6} \frac{\eta^{5l_d}(2 d z)}{\eta^{2l_d}( d z) \eta^{2 l_d}(4 d z)} \label{9_4} \\
%&&= \frac{\eta^{5l_0}(2 z) \eta^{5l_1}(6 z) \eta^{5l_2}(18 z) \eta^{5l_3}(54 z)}{\eta^{2l_0}( z) \eta^{2 l_0}(4 z) \eta^{2l_1}( 3 z) \eta^{2 l_1}(12 z) \eta^{2l_2}( 9 z) \eta^{2 l_2}(36 z) \eta^{2l_3}( 27 z) \eta^{2 l_3}(108 z)}\\
&=\eta_{24}[-2l_1, 5l_1-2l_2, -2l_3 , -2l_1+5l_2, 5l_3-2l_6 , -2l_2 , -2l_3+5l_6 , -2l_6 ](z). \nonumber
\end{align}
By \cite[Proposition 2.6]{iwaniec} a set of representatives of all cusps of $\Gamma_0(24)$ can be chosen as follows:
\beqars
R(\Gamma_0(24))=\left\{ 1,\frac{1}{2},\frac{1}{3},\frac{1}{4},\frac{1}{6},\frac{1}{8},\frac{1}{12},\frac{1}{24} \right\}.
\eeqars
We use (\ref{etaqorder}) to compute the orders of (\ref{9_4}) at each cusp $r \in R(\Gamma_0(24))$:
\beqars
&& v_{1/c}\left(\prod_{d \mid 6} \varphi^{l_d}(d z)\right)= 0 \mbox{ for $c =1,3,8,24$},\\
&& v_{1/2}\left(\prod_{d \mid 6} \varphi^{l_d}(d z)\right)=\frac{3}{2}l_1+\frac{1}{2}l_3 \geq 0,\qquad v_{1/4}\left(\prod_{d \mid 6} \varphi^{l_d}(d z)\right)= \frac{3}{2}l_2+\frac{1}{2}l_6 \geq 0,\\
&& v_{1/6}\left(\prod_{d \mid 6} \varphi^{l_d}(d z)\right)= \frac{1}{2}l_1+\frac{3}{2}l_3 \geq 0,\qquad v_{1/12}\left(\prod_{d \mid 6} \varphi^{l_d}(d z)\right)= \frac{1}{2}l_2+\frac{3}{2}l_6 \geq 0.
\eeqars
So by Theorem \ref{th:4_30}, $\ds \prod_{d \mid 6} \varphi^{l_d}(d z)$ is in $M_{3}(\Gamma_0(24), \chi)$, where, by appealing to (\ref{char}), we have
\beqars
\chi = \left\{\begin{array}{ll}
\chi_{_{-4}} & \mbox{ if $ l_1+l_3 $ is even and $ l_3+l_6 $ is even, } \\
\chi_{_{-3}} & \mbox{ if $ l_1+l_3 $ is even and $ l_3+l_6 $ is odd, } \\
\chi_{_{-8}} & \mbox{ if $ l_1+l_3 $ is odd and $ l_3+l_6 $ is even, } \\
\chi_{_{-24}} & \mbox{ if $ l_1+l_3 $ is odd and $ l_3+l_6 $ is odd. } \\
\end{array} \right.
\eeqars
\end{proof} 

\section{Bases for $M_3(\Gamma_0(24),\chi)$} \label{sec:8_1}
In this section we give bases for $M_3(\Gamma_0(24), \chi)$ ($\chi= \chi_{_{-3}},\chi_{_{-4}},\chi_{_{-8}},\chi_{_{-24}}$) in terms of Eisenstein series and eta quotients. 

\begin{thm} \label{th:8_1} {\bf (i)}The set of Eisenstein series
\beqars
&E(3,24,\chi_{_{-3}})= & \{E_{3,\chi_{_{-3}},\chi_{_{1}}}(dz), E_{3,\chi_{_{1}},\chi_{_{-3}}}(dz) \mid d =1,2,4, 8 \} 
\eeqars
constitute a basis for $E_3(\Gamma_0(24),\chi_{_{-3}})$. \\
{\bf (ii)} The ordered set of eta quotients
\begin{align*}
S(3,24,\chi_{_{-3}})=& \{ \eta_{24}[0,3,0,-4,-5,2,16,-6](z), \eta_{24}[1,-1,-3,1,7,0,1,0](z), \\
& \hspace{3mm} \eta_{24}[0,2,0,-1,-2,0,7,0](z), \eta_{24}[0,1,0,2,1,-2,-2,6](z)\}
\end{align*}
constitute a basis for $S_3(\Gamma_0(24),\chi_{_{-3}})$. \\
{\bf (iii)} The set $E(3,24,\chi_{_{-3}}) \cup S(3,24,\chi_{_{-3}})$ constitute a basis for $M_3(\Gamma_0(24),\chi_{_{-3}})$.
\end{thm}
\begin{proof} 
Let $\varepsilon$ and $\psi$ be primitive characters with conductors $R$ and $L$, respectively. By \cite[Theorem 5.9]{stein}, the set of Eisenstein series
\beqars
 \{ E_{2,\varepsilon,\psi}(tz) \mid \varepsilon \psi = \chi_{_{-3}},~ RLt \mid 24 \} 
\eeqars
gives a basis for $E_3(\Gamma_0(24),\chi_{_{-3}})$. All (primitive) characters with conductors dividing $24$ and their values at numbers coprime to 24 are given in the following table, where $C$ denotes conductor of the corresponding character.
\begin{longtable}{ l r r r r r r r r r } 
\hline \hline
 $$ & $C$ & $1$ &$5$ & $7$ & $11$&$13$ &$17$ & $19$ & $23$ \\ 
\hline \hline
\endfirsthead
%{\textit{Continued from previous page}} \\
\hline
\hline 
 $$ & $C$ & $1$ &$5$ & $7$ & $11$&$13$ &$17$ & $19$ & $23$ \\ 
\hline \hline
\endhead
\hline % {\textit{Continued on next page}} \\
\endfoot
\hline
\endlastfoot
 $\chi_{_1}$ & $1$ & $1$ & $1$ & $1$ & $1$ & $1$ & $1$ & $1$ & $1$ \\ 
$\chi_{_{-24}}$ & $24$ & $1$ & $1$ & $1$ & $1$ & $-1$ & $-1$ & $-1$ & $-1$ \\ 
$\chi_{_{-4}}$ & $4$ & $1$ & $1$ & $-1$ & $-1$ & $1$ & $1$ & $-1$ & $-1$ \\ 
 $\chi_{_{24}}$ & $24$ & $1$ & $1$ & $-1$ & $-1$ & $-1$ & $-1$ & $1$ & $1$ \\ 
 $\chi_{_{8}}$ & $8$& $1$ & $-1$ & $1$ & $-1$ & $-1$ & $1$ & $-1$ & $1$ \\ 
 $\chi_{_{-3}}$ & $3$ & $1$ & $-1$ & $1$ & $-1$ & $1$ & $-1$ & $1$ & $-1$ \\ 
 $\chi_{_{-8}}$ &$8$ & $1$ & $-1$ & $-1$ & $1$ & $-1$ & $1$ & $1$ & $-1$ \\ 
 $\chi_{_{12}}$ & $12$ & $1$ & $-1$ & $-1$ & $1$ & $1$ & $-1$ & $-1$ & $1$ \\ 
\end{longtable}
From the table we calculate
\beqars
& \chi_{_{1}} \chi_{-3}=\chi_{_{-3}}, & \chi_{_{-3}} \chi_{_{1}}=\chi_{_{-3}} \mbox{ with $RL=3 \mid 24$,}\\
& \chi_{_{8}} \chi_{-24}=\chi_{_{-3}}, & \chi_{_{-24}} \chi_{_{8}}=\chi_{_{-3}} \mbox{ with $RL=192 \nmid 24$,}\\
& \chi_{_{-8}} \chi_{24}=\chi_{_{-3}}, & \chi_{_{24}} \chi_{_{-8}}=\chi_{_{-3}} \mbox{ with $RL=192 \nmid 24$,}\\
& \chi_{_{-4}} \chi_{12}=\chi_{_{-3}}, & \chi_{_{12}} \chi_{_{-4}}=\chi_{_{-3}} \mbox{ with $RL=48 \nmid 24$.}
\eeqars
That is $\{ E_{2,\varepsilon,\psi}(tz) \mid \varepsilon \psi = \chi_{_{-3}},~ RLt \mid 24 \} = E(3,24,\chi_{_{-3}})$. Thus we deduce that the set of Eisenstein series given by $E(3,24,\chi_{_{-3}})$ constitute a basis for $E_3(\Gamma_0(24),\chi_{_{-3}})$.

Now we prove assertion (ii) of the theorem. From \cite[Section 6.3]{stein} we obtain
\beqars
&& \dim(S_3(\Gamma_0(24),\chi_{_{-3}}))=4.
\eeqars
Let $f(z)=\eta_{24}[0,3,0,-4,-5,2,16,-6](z)$. By (\ref{etaqorder}), for $r\in R(\Gamma_0(24))$ we have
\beqars
v_r(f(z))= \left\{\begin{array}{ll}
5, & \mbox{ if $r=\frac{1}{12}$} , \\
1, & \mbox{ otherwise}.\\
\end{array} \right.
\eeqars
So, by Theorem \ref{th:4_30} and (\ref{char}), $f(z)\in S_3(\Gamma_0(24),\chi_{_{-3}})$. Similarly we show that for all eta quotients $f(z) \in S(3,24,\chi_{_{-3}})$, we have $f(z) \in S_3(\Gamma_0(24),\chi_{_{-3}})$. The orders of all the eta quotients in $S(3,24,\chi_{_{-3}})$ at the cusp $\ds \frac{1}{24}$ (or equivalently at $\infty$) are different. Thus eta quotients given in $ S(3,24,\chi_{_{-3}})$ are linearly independent, which completes the proof of part (ii).

Part (iii) of theorem follows from (\ref{1_1}).
\end{proof}
The rest of the theorems in this section can be proven similarly.
\begin{thm}\label{th:8_2} {\bf (i)}The set of Eisenstein series
\beqars
&E(3,24,\chi_{_{-4}})= & \{E_{3,\chi_{_{-4}},\chi_{_{1}}}(dz), E_{3,\chi_{_{1}},\chi_{_{-4}}}(dz) \mid d=1,2,3, 6 \} 
\eeqars
constitute a basis for $E_3(\Gamma_0(24),\chi_{_{-4}})$. \\
{\bf (ii)} The ordered set of eta quotients
\begin{align*}
S(3,24,\chi_{_{-4}})=& \{ \eta_{24}[1,-1,-3,0,7,2,2,-2](z),\eta_{24}[0,2,0,-2,-2,2,8,-2](z), \\
& ~ \eta_{24}[0,0,0,4,0,-2,2,2](z), \eta_{24}[0,1,0,1,1,0-1,4](z)\}
\end{align*}
constitute a basis for $S_3(\Gamma_0(24),\chi_{_{-4}})$. \\
{\bf (iii)} The set $E(3,24,\chi_{_{-4}}) \cup S(3,24,\chi_{_{-4}})$ constitute a basis for $M_3(\Gamma_0(24),\chi_{_{-4}})$.
\end{thm}

\begin{thm} {\bf (i)}The set of Eisenstein series
\beqars
&E(3,24,\chi_{_{-8}})= & \{E_{3,\chi_{_{-8}},\chi_{_{1}}}(dz), E_{3,\chi_{_{1}},\chi_{_{-8}}}(dz) \mid d =1, 3 \}
\eeqars
constitute a basis for $E_3(\Gamma_0(24),\chi_{_{-8}})$. \\
{\bf (ii)} The ordered set of eta quotients
\begin{align*}
S(3,24,\chi_{_{-8}})=& \{ \eta_{24}[2,-2,-4,-2,7,2,7,-4](z), \eta_{24}[1,1,-1,-4,-2,2,13,-4](z), \\
& ~\eta_{24}[2,-3,-4,1,10,0,-2,2](z), \eta_{24}[1,0,-1,-1,1,0,4,2](z), \\
& ~\eta_{24}[0,1,2,0,-2,-1,1,5](z),\eta_{24}[1,-1,-1,2,4,-2,-5,8](z)\}
\end{align*}
constitute a basis for $S_3(\Gamma_0(24),\chi_{_{-8}})$. \\
{\bf (iii)} The set $E(3,24,\chi_{_{-8}}) \cup S(3,24,\chi_{_{-8}})$ constitute a basis for $M_3(\Gamma_0(24),\chi_{_{-8}})$.
\end{thm}

\begin{thm}\label{th:1} {\bf (i)}The set of Eisenstein series
\beqars
&E(3,24,\chi_{_{-24}})= & \{E_{3,\chi_{_{-24}},\chi_{_{1}}}(z), E_{3,\chi_{_{1}},\chi_{_{-24}}}(z), E_{3,\chi_{_{-3}},\chi_{_{8}}}(z), E_{3,\chi_{_{8}},\chi_{_{-3}}}(z) \}
\eeqars
constitute a basis for $E_3(\Gamma_0(24),\chi_{_{-24}})$. \\
{\bf (ii)} The ordered set of eta quotients
\begin{align*}
S(3,24,\chi_{_{-24}})=& \{ \eta_{24}[1,1,-1,-5,-2,4,14,-6](z),\eta_{24}[2,-3,-4,0,10,2,-1,0](z) \\
&~ \eta_{24}[1,0,-1,-2,1,2,5,0](z), \eta_{24}[1,-2,-1,4,3,-2,-1,4](z), \\
& ~\eta_{24}[1,-1,-1,1,4,0,-4,6](z), \eta_{24}[-1,4,1,0,-1,-2,-3,8](z)\}
\end{align*}
constitute a basis for $S_3(\Gamma_0(24),\chi_{_{-24}})$. \\
{\bf (iii)} The set $E(3,24,\chi_{_{-24}}) \cup S(3,24,\chi_{_{-24}})$ constitute a basis for $M_3(\Gamma_0(24),\chi_{_{-24}})$.
\end{thm}

\section{Main Results: Sextenary Quadratic Forms}\label{sec:9_1}
In Section \ref{sec:2}, we determined that the generating functions of (\ref{9_3}) are modular forms, whose corresponding spaces are given by Theorem \ref{th:9_1}. In Section \ref{sec:8_1} we constructed the bases for all these modular form spaces. In this section we use those bases obtained to state Theorems \ref{th:9_2} and \ref{th:9_3}, which combined with the Tables \ref{table:9_1}--\ref{table:9_4}, give the desired formulas for the numbers of representations of positive integers by diagonal sextenary quadratic forms with coefficients $1$, $2$, $3$ and $6$. Note that $S_j$ stands for the $j$th element in the ordered set $S$.

\begin{thm} \label{th:9_2}
Let $l_d \in \nn_0 $ {\rm ($d \mid 6$)} and $ l_1+l_2+l_3+l_6=6$. Then we have
\begin{align}
& \prod_{d \mid 6} \varphi^{l_d}(d z) \nonumber \\
& = \left\{\begin{array}{l}
\ds \sum_{t \mid 6} a_t E_{3,\chi_{_{-4}},\chi_{_{1}}}(tz)+\sum_{t \mid 6} b_t E_{3,\chi_{_{1}},\chi_{_{-4}}}(tz)+\sum_{j=1}^{4} c_j S(3,24,\chi_{_{-4}})_j \\
\ds \qquad\qquad\qquad\qquad\qquad \mbox{if $ l_1+l_3 \equiv 0 \pmd{2} $ and $ l_3+l_6 \equiv 0 \pmd{2} $,} \\
\ds \sum_{t \mid 8} d_t E_{3,\chi_{_{-3}},\chi_{_{1}}}(tz)+ \sum_{t \mid 8} e_t E_{3,\chi_{_{1}},\chi_{_{-3}}}(tz)+\sum_{j=1}^{4} f_j S(3,24,\chi_{_{-3}})_j \\
\ds \qquad\qquad\qquad\qquad\qquad \mbox{if $ l_1+l_3 \equiv 0 \pmd{2} $ and $ l_3+l_6 \equiv 1 \pmd{2} $, } \\
\ds \sum_{t \mid 3} g_t E_{3,\chi_{_{-8}},\chi_{_{1}}}(tz)+\sum_{t \mid 3} h_t E_{3,\chi_{_{1}},\chi_{_{-8}}}(tz)+\sum_{j=1}^{6} k_j S(3,24,\chi_{_{-8}})_j\\
\ds \qquad\qquad\qquad\qquad\qquad \mbox{if $ l_1+l_3 \equiv 1 \pmd{2} $ and $ l_3+l_6 \equiv 0 \pmd{2} $,} \\
\ds m_1 E_{3,\chi_{_{-24}},\chi_{_{1}}}(z) + m_2 E_{3,\chi_{_{1}},\chi_{_{-24}}}(z)+ m_3 E_{3,\chi_{_{-3}},\chi_{_{8}}}(z) +m_4 E_{3,\chi_{_{8}},\chi_{_{-3}}}(z) \\
\qquad +\sum_{j=1}^{6} n_j S(3,24,\chi_{_{-24}})_j\\
\ds \qquad\qquad\qquad\qquad\qquad \mbox{if $ l_1+l_3 \equiv 1 \pmd{2} $ and $ l_3+l_6 \equiv 1 \pmd{2} $, } 
\end{array} \right. \label{9_1}
\end{align}
where the values $a_t$, $b_t$, $c_j$; $d_t$, $e_t$, $f_j$; $g_t$, $h_t$, $k_j$; $m_i$, $n_j$ are given in {\em Tables \ref{table:9_1} -- \ref{table:9_4}}, respectively.
\end{thm}
\begin{proof}
Appealing to (\ref{1_1}) and Theorems \ref{th:8_1}--\ref{th:1} we deduce the linear combinations given by (\ref{9_1}). We determine the values of $a_t$, $b_t$, $c_j$, $d_t$, $e_t$, $f_j$, $g_t$, $h_t$, $k_j$, $m_i$, $n_j$ by comparing the first few coefficients of the Fourier series expansions of both sides in (\ref{9_1}). Solutions to resulting equations generate Tables \ref{table:9_1}--\ref{table:9_4}. We use MAPLE to perform the calculations.
\end{proof}
We compare the coefficients of $q^n$ ($n>0$) in the equations given by Theorem \ref{th:9_2} to obtain the following theorem.
\begin{thm}\label{th:9_3}
Let $l_d \in \nn_0 $ {\rm ($d \mid 6$)} and $ l_1+l_2+l_3+l_6=6$. Then for $n>0$ we have
\begin{align*}
& N(1^{l_1},2^{l_2},3^{l_3},6^{l_6}; n) \\
& = \left\{\begin{array}{l}
\ds \sum_{t \mid 6} a_t \sigma_{(2,\chi_{_{-4}},\chi_{_{1}})}(n/t)+\sum_{t \mid 6} b_t \sigma_{(2,\chi_{_{1}},\chi_{_{-4}})}(n/t) +\sum_{j=1}^{4} c_j [n] S(3,24,\chi_{_{-4}})_j \\
\ds \qquad\qquad\qquad\qquad\qquad\qquad \mbox{if $ l_1+l_3 \equiv 0 \pmd{2} $ and $ l_3+l_6 \equiv 0 \pmd{2} $,} \\
\ds \sum_{t \mid 8} d_t \sigma_{(2,\chi_{_{-3}},\chi_{_{1}})}(n/t)+ \sum_{t \mid 8} e_t \sigma_{(2,\chi_{_{1}},\chi_{_{-3}})}(n/t)+\sum_{j=1}^{4} f_j [n] S(3,24,\chi_{_{-3}})_j \\
\ds \qquad\qquad\qquad\qquad\qquad\qquad \mbox{if $ l_1+l_3 \equiv 0 \pmd{2} $ and $ l_3+l_6 \equiv 1 \pmd{2} $, } \\
\ds \sum_{t \mid 3} g_t \sigma_{(2,\chi_{_{-8}},\chi_{_{1}})}(n/t)+\sum_{t \mid 3} h_t \sigma_{(2,\chi_{_{1}},\chi_{_{-8}})}(n/t)+\sum_{j=1}^{6} k_j [n] S(3,24,\chi_{_{-8}})_j \\
\ds \qquad\qquad\qquad\qquad\qquad\qquad \mbox{if $ l_1+l_3 \equiv 1 \pmd{2} $ and $ l_3+l_6 \equiv 0 \pmd{2} $,} \\
\ds m_1 \sigma_{(2,\chi_{_{-24}},\chi_{_{1}})}(n) + m_2 \sigma_{(2,\chi_{_{1}},\chi_{_{-24}})}(n)+ m_3 \sigma_{(2,\chi_{_{-3}},\chi_{_{8}})}(n) +m_4 \sigma_{(2,\chi_{_{8}},\chi_{_{-3}})}(n) \\
\ds \qquad +\sum_{j=1}^{6} n_j [n] S(3,24,\chi_{_{-24}})_j\\
\ds \qquad\qquad\qquad\qquad\qquad\qquad \mbox{if $ l_1+l_3 \equiv 1 \pmd{2} $ and $ l_3+l_6 \equiv 1 \pmd{2} $, } 
\end{array} \right.
\end{align*}
where the values $a_t$, $b_t$, $c_j$; $d_t$, $e_t$, $f_j$; $g_t$, $h_t$, $k_j$; $m_i$, $n_j$ are given in {\em Tables \ref{table:9_1} -- \ref{table:9_4}}, respectively.
\end{thm}

%\begin{landscape}
{\small \begin{center}
\renewcommand*{\arraystretch}{1.2} \begin{longtable}{r r r r | c c c c c c c c c c c c  }
\caption{Values of $l_d$, $a_t$, $b_t$, $c_j$ for Theorems \ref{th:9_2} and \ref{th:9_3} } \label{table:9_1} \\
\hline\hline 
$l_1$ & $l_2$ & $l_3$ & $l_6$ & $a_1$ & $a_2$ & $a_3$ & $a_6$ & $b_1$ & $b_2$ & $b_3$ & $b_6$ & $c_1$ & $c_2$ & $c_3$ & $c_4$ \\ 
\hline
\endfirsthead
\multicolumn{16}{c}
{\tablename\ \thetable\ -- \textit{Continued from previous page}} \\
\hline
\hline 
$l_1$ & $l_2$ & $l_3$ & $l_6$ & $a_1$ & $a_2$ & $a_3$ & $a_6$ & $b_1$ & $b_2$ & $b_3$ & $b_6$ & $c_1$ & $c_2$ & $c_3$ & $c_4$ \\ 
\hline
\endhead
\hline \multicolumn{16}{r}{\textit{Continued on next page}} \\
\endfoot
\hline
\endlastfoot
$6$ & $ 0$ & $ 0$ & $ 0$ & $ -4$ & $ 0$ & $ 0$ & $ 0$ & $ 16$ & $ 0$ & $ 0$ & $ 0$ & $ 0$ & $ 0$ & $ 0$ & $ 0$ \\
 $4$ & $ 2$ & $ 0$ & $ 0$ & $ 0$ & $ -4$ & $ 0$ & $ 0$ & $ 8$ & $ 0$ & $ 0$ & $ 0$ & $ 0$ & $ 0$ & $ 0$ & $ 0$ \\
 $4$ & $ 0$ & $ 2$ & $ 0$ & $  \frac{8}{7}$ & $ 0$ & $ \frac{-36}{7}$ & $ 0$ & $  \frac{32}{7}$ & $ 0$ & $  \frac{144}{7}$ & $ 0$ & $  \frac{16}{7}$ & $  \frac{48}{7}$ & $ \frac{-48}{7}$ & $ \frac{-64}{7}$ \\
 $4$ & $ 0$ & $ 0$ & $ 2$ & $ 0$ & $  \frac{8}{7}$ & $ 0$ & $ \frac{-36}{7}$ & $  \frac{16}{7}$ & $ 0$ & $  \frac{72}{7}$ & $ 0$ & $  \frac{40}{7}$ & $  \frac{136}{7}$ & $  \frac{24}{7}$ & $  \frac{96}{7}$ \\
 $3$ & $ 1$ & $ 1$ & $ 1$ & $ 0$ & $ \frac{-10}{7}$ & $ 0$ & $ \frac{-18}{7}$ & $  \frac{20}{7}$ & $ 0$ & $ \frac{-36}{7}$ & $ 0$ & $  \frac{22}{7}$ & $  \frac{50}{7}$ & $  \frac{30}{7}$ & $  \frac{40}{7}$ \\
 $2$ & $ 4$ & $ 0$ & $ 0$ & $ 0$ & $ -4$ & $ 0$ & $ 0$ & $ 4$ & $ 0$ & $ 0$ & $ 0$ & $ 0$ & $ 0$ & $ 0$ & $ 0$ \\
 $2$ & $ 2$ & $ 2$ & $ 0$ & $ 0$ & $  \frac{8}{7}$ & $ 0$ & $ \frac{-36}{7}$ & $  \frac{16}{7}$ & $ 0$ & $  \frac{72}{7}$ & $ 0$ & $  \frac{12}{7}$ & $ \frac{-4}{7}$ & $ \frac{-60}{7}$ & $ \frac{-16}{7}$ \\
 $2$ & $ 2$ & $ 0$ & $ 2$ & $ 0$ & $  \frac{8}{7}$ & $ 0$ & $ \frac{-36}{7}$ & $  \frac{8}{7}$ & $ 0$ & $  \frac{36}{7}$ & $ 0$ & $  \frac{20}{7}$ & $  \frac{36}{7}$ & $ \frac{ 12}{7}$ & $  \frac{64}{7}$ \\
 $2$ & $ 0$ & $ 4$ & $ 0$ & $ \frac{-4}{7}$ & $ 0$ & $ \frac{-24}{7}$ & $ 0$ & $  \frac{16}{7}$ & $ 0$ & $ \frac{-96}{7}$ & $ 0$ & $  \frac{16}{7}$ & $ \frac{ -16}{7}$ & $  \frac{16}{7}$ & $ \frac{-64}{7}$ \\
 $2$ & $ 0$ & $ 2$ & $ 2$ & $ 0$ & $ \frac{-4}{7}$ & $ 0$ & $ \frac{-24}{7}$ & $  \frac{8}{7}$ & $ 0$ & $ \frac{-48}{7}$ & $ 0$ & $  \frac{20}{7}$ & $  \frac{20}{7}$ & $  \frac{12}{7}$ & $  \frac{16}{7}$ \\
 $2$ & $ 0$ & $ 0$ & $ 4$ & $ 0$ & $ \frac{-4}{7}$ & $ 0$ & $ \frac{-24}{7}$ & $  \frac{4}{7}$ & $ 0$ & $ \frac{-24}{7}$ & $ 0$ & $  \frac{24}{7}$ & $ \frac{ 40}{7}$ & $ \frac{ -8}{7}$ & $ 0$ \\
 $1$ & $ 3$ & $ 1$ & $ 1$ & $ 0$ & $ \frac{-10}{7}$ & $ 0$ & $ \frac{-18}{7}$ & $  \frac{10}{7}$ & $ 0$ & $ \frac{-18}{7}$ & $ 0$ & $  \frac{4}{7}$ & $  \frac{16}{7}$ & $  \frac{36}{7}$ & $ 0$ \\
 $1$ & $ 1$ & $ 3$ & $ 1$ & $ 0$ & $  \frac{2}{7}$ & $ 0$ & $ \frac{-30}{7}$ & $  \frac{4}{7}$ & $ 0$ & $  \frac{60}{7}$ & $ 0$ & $  \frac{10}{7}$ & $ \frac{ 6}{7}$ & $ \frac{-22}{7}$ & $  \frac{24}{7}$ \\
 $1$ & $ 1$ & $ 1$ & $ 3$ & $ 0$ & $  \frac{2}{7}$ & $ 0$ & $ \frac{-30}{7}$ & $  \frac{2}{7}$ & $ 0$ & $  \frac{30}{7}$ & $ 0$ & $  \frac{12}{7}$ & $  \frac{16}{7}$ & $ \frac{-4}{7}$ & $  \frac{16}{7}$ \\
 $0$ & $ 6$ & $ 0$ & $ 0$ & $ 0$ & $ -4$ & $ 0$ & $ 0$ & $ 0$ & $ 16$ & $ 0$ & $ 0$ & $ 0$ & $ 0$ & $ 0$ & $ 0$ \\
 $0$ & $ 4$ & $ 2$ & $ 0$ & $ 0$ & $  \frac{8}{7}$ & $ 0$ & $ \frac{-36}{7}$ & $  \frac{8}{7}$ & $ 0$ & $  \frac{36}{7}$ & $ 0$ & $ \frac{-8}{7}$ & $  \frac{8}{7}$ & $ \frac{-72}{7}$ & $  \frac{64}{7}$ \\
 $0$ & $ 4$ & $ 0$ & $ 2$ & $ 0$ & $  \frac{8}{7}$ & $ 0$ & $ \frac{-36}{7}$ & $ 0$ & $  \frac{32}{7}$ & $ 0$ & $  \frac{144}{7}$ & $ 0$ & $  \frac{16}{7}$ & $ 0$ & $  \frac{64}{7}$ \\
 $0$ & $ 2$ & $ 4$ & $ 0$ & $ 0$ & $ \frac{-4}{7}$ & $ 0$ & $ \frac{-24}{7}$ & $  \frac{8}{7}$ & $ 0$ & $ \frac{-48}{7}$ & $ 0$ & $ \frac{-8}{7}$ & $ \frac{-8}{7}$ & $  \frac{40}{7}$ & $ \frac{-96}{7}$ \\
 $0$ & $ 2$ & $ 2$ & $ 2$ & $ 0$ & $ \frac{-4}{7}$ & $ 0$ & $ \frac{-24}{7}$ & $  \frac{4}{7}$ & $ 0$ & $ \frac{-24}{7}$ & $ 0$ & $ \frac{-4}{7}$ & $  \frac{12}{7}$ & $  \frac{20}{7}$ & $ 0$ \\
 $0$ & $ 2$ & $ 0$ & $ 4$ & $ 0$ & $ \frac{-4}{7}$ & $ 0$ & $ \frac{-24}{7}$ & $ 0$ & $  \frac{16}{7}$ & $ 0$ & $ \frac{-96}{7}$ & $ 0$ & $  \frac{16}{7}$ & $ 0$ & $ 0$ \\
 $0$ & $ 0$ & $ 6$ & $ 0$ & $ 0$ & $ 0$ & $ -4$ & $ 0$ & $ 0$ & $ 0$ & $ 16$ & $ 0$ & $ 0$ & $ 0$ & $ 0$ & $ 0$ \\
 $0$ & $ 0$ & $ 4$ & $ 2$ & $ 0$ & $ 0$ & $ 0$ & $ -4$ & $ 0$ & $ 0$ & $ 8$ & $ 0$ & $ 0$ & $ 0$ & $ 0$ & $ 0$ \\
 $0$ & $ 0$ & $ 2$ & $ 4$ & $ 0$ & $ 0$ & $ 0$ & $ -4$ & $ 0$ & $ 0$ & $ 4$ & $ 0$ & $ 0$ & $ 0$ & $ 0$ & $ 0$ \\
 $0$ & $ 0$ & $ 0$ & $ 6$ & $ 0$ & $ 0$ & $ 0$ & $ -4$ & $ 0$ & $ 0$ & $ 0$ & $ 16$ & $ 0$ & $ 0$ & $ 0$ & $ 0$
\end{longtable}
\end{center} }
%\end{landscape}

%\begin{landscape}
{\small \begin{center}
\renewcommand*{\arraystretch}{1.2} \begin{longtable}{r r r r | c c c c c c c c c c c c }
\caption{Values of $l_d$, $d_t$, $e_t$, $f_j$ for Theorems \ref{th:9_2} and \ref{th:9_3} }\label{table:9_2} \\
\hline\hline 
$l_1$ & $l_2$ & $l_3$ & $l_6$ & $d_1$ & $d_2$ & $d_{4}$ & $d_8$ & $e_1$ & $e_2$ & $e_4$ & $e_8$ & $f_1$ & $f_2$ & $f_3$ & $f_4$ \\ 
\hline
\endfirsthead
\multicolumn{16}{c}
{\tablename\ \thetable\ -- \textit{Continued from previous page}} \\
\hline
\hline 
$l_1$ & $l_2$ & $l_3$ & $l_6$ & $d_1$ & $d_2$ & $d_{4}$ & $d_8$ & $e_1$ & $e_2$ & $e_4$ & $e_8$ & $f_1$ & $f_2$ & $f_3$ & $f_4$ \\ 
\hline
\endhead
\hline \multicolumn{16}{r}{\textit{Continued on next page}} \\
\endfoot
\hline
\endlastfoot
$5$ & $ 0$ & $ 1$ & $ 0$ & $ 1$ & $ -2$ & $ -8$ & $ 0$ & $ 9$ & $ 18$ & $ -72$ & $ 0$ & $ 0$ & $ 0$ & $ 0$ & $ 0$ \\
 $4$ & $ 1$ & $ 0$ & $ 1$ & $ \frac{-1}{2}$ & $ \frac{1}{2}$ & $ -1$ & $ -8$ & $ \frac{9}{2}$ & $ \frac{9}{2}$ & $ 9$ & $ -72$ & $ 4$ & $ 6$ & $ 26$ & $ 16$ \\
 $3$ & $ 2$ & $ 1$ & $ 0$ & $ \frac{1}{2}$ & $ \frac{-1}{2}$ & $ -1$ & $ -8$ & $ \frac{9}{2}$ & $ \frac{9}{2}$ & $ -9$ & $ 72$ & $ 1$ & $ 0$ & $ -4$ & $ 4$ \\
 $3$ & $ 0$ & $ 3$ & $ 0$ & $ -1$ & $ 0$ & $ -8$ & $ 0$ & $ 3$ & $ 0$ & $ 24$ & $ 0$ & $ 4$ & $ 0$ & $ 0$ & $ -16$ \\
 $3$ & $ 0$ & $ 1$ & $ 2$ & $ \frac{-1}{2}$ & $ \frac{1}{2}$ & $ -1$ & $ -8$ & $ \frac{3}{2}$ & $ \frac{3}{2}$ & $ 3$ & $ -24$ & $ 5$ & $ 4$ & $ 16$ & $ 4$ \\
 $2$ & $ 3$ & $ 0$ & $ 1$ & $ \frac{-1}{4}$ & $ \frac{1}{4}$ & $ -1$ & $ -8$ & $ \frac{9}{4}$ & $ \frac{9}{4}$ & $ 9$ & $ -72$ & $ 2$ & $ 0$ & $ 10$ & $ 8$ \\
 $2$ & $ 1$ & $ 2$ & $ 1$ & $ \frac{1}{2}$ & $ \frac{-1}{2}$ & $ -1$ & $ -8$ & $ \frac{3}{2}$ & $\frac{ 3}{2}$ & $ -3$ & $ 24$ & $ 2$ & $ 2$ & $ 6$ & $ 8$ \\
 $2$ & $ 1$ & $ 0$ & $ 3$ & $ \frac{1}{4}$ & $ \frac{-1}{4}$ & $ -1$ & $ -8$ & $ \frac{3}{4}$ & $ \frac{3}{4}$ & $ -3$ & $ 24$ & $ 3$ & $ 4$ & $ 14$ & $ 4$ \\
 $1$ & $ 4$ & $ 1$ & $ 0$ & $ \frac{1}{4}$ & $ \frac{-1}{4}$ & $ -1$ & $ -8$ & $ \frac{9}{4}$ & $ \frac{9}{4}$ & $ -9$ & $ 72$ & $ \frac{-1}{2}$ & $ 0$ & $ -4$ & $ 10$ \\
 $1$ & $ 2$ & $ 3$ & $ 0$ & $ \frac{-1}{2}$ & $ \frac{1}{2}$ & $ -1$ & $ -8$ & $ \frac{3}{2}$ & $ \frac{3}{2}$ & $ 3$ & $ -24$ & $ 1$ & $ -4$ & $ 0$ & $ -12$ \\
 $1$ & $ 2$ & $ 1$ & $ 2$ & $ \frac{-1}{4}$ & $ \frac{1}{4}$ & $ -1$ & $ -8$ & $ \frac{3}{4}$ & $ \frac{3}{4}$ & $ 3$ & $ -24$ & $\frac{ 3}{2}$ & $ 0$ & $ 8$ & $ 2$ \\
 $1$ & $ 0$ & $ 5$ & $ 0$ & $ 1$ & $ -2$ & $ -8$ & $ 0$ & $ 1$ & $ 2$ & $ -8$ & $ 0$ & $ 0$ & $ 0$ & $ 0$ & $ 0$ \\
 $1$ & $ 0$ & $ 3$ & $ 2$ & $ \frac{1}{2}$ & $ \frac{-1}{2}$ & $ -1$ & $ -8$ & $ \frac{1}{2}$ & $ \frac{1}{2}$ & $ -1$ & $ 8$ & $ 1$ & $ 0$ & $ 4$ & $ 4$ \\
 $1$ & $ 0$ & $ 1$ & $ 4$ & $ \frac{1}{4}$ & $ \frac{-1}{4}$ & $ -1$ & $ -8$ & $ \frac{1}{4}$ & $ \frac{1}{4}$ & $ -1$ & $ 8$ & $ \frac{3}{2}$ & $ 0$ & $ 4$ & $ 2$ \\
 $0$ & $ 5$ & $ 0$ & $ 1$ & $ 0$ & $ 1$ & $ -2$ & $ -8$ & $ 0$ & $ 9$ & $ 18$ & $ -72$ & $ 0$ & $ 0$ & $ 0$ & $ 0$ \\
 $0$ & $ 3$ & $ 2$ & $ 1$ & $ \frac{1}{4}$ & $ \frac{-1}{4}$ & $ -1$ & $ -8$ & $ \frac{3}{4}$ & $ \frac{3}{4}$ & $ -3$ & $ 24$ & $ -1$ & $ 4$ & $ -2$ & $ 4$ \\
 $0$ & $ 3$ & $ 0$ & $ 3$ & $ 0$ & $ -1$ & $ 0$ & $ -8$ & $ 0$ & $ 3$ & $ 0$ & $ 24$ & $ 0$ & $ 4$ & $ 4$ & $ 0$ \\
 $0$ & $ 1$ & $ 4$ & $ 1$ & $ \frac{-1}{2}$ & $ \frac{1}{2}$ & $ -1$ & $ -8$ & $ \frac{1}{2}$ & $ \frac{1}{2}$ & $ 1$ & $ -8$ & $ 0$ & $ -2$ & $ 2$ & $ 0$ \\
 $0$ & $ 1$ & $ 2$ & $ 3$ & $ \frac{-1}{4}$ & $ \frac{1}{4}$ & $ -1$ & $ -8$ & $ \frac{1}{4}$ & $ \frac{1}{4}$ & $ 1$ & $ -8$ & $ 0$ & $ 0$ & $ 2$ & $ 0$ \\
 $0$ & $ 1$ & $ 0$ & $ 5$ & $ 0$ & $ 1$ & $ -2$ & $ -8$ & $ 0$ & $ 1$ & $ 2$ & $ -8$ & $ 0$ & $ 0$ & $ 0$ & $ 0$
\end{longtable}
\end{center} }
%\end{landscape}

%\begin{landscape}
{\small \begin{center}
\renewcommand*{\arraystretch}{1.2} \begin{longtable}{r r r r | c c c c c c c c c c }
\caption{ Values of $l_d$, $g_t$, $h_{t}$, $k_j$ for Theorems \ref{th:9_2} and \ref{th:9_3} } \label{table:9_3} \\
\hline\hline 
$l_1$ & $l_2$ & $l_3$ & $l_6$ & $g_1$ & $g_3$ & $h_{1}$ & $h_3$ & $k_1$ & $k_2$ & $k_3$ & $k_4$ & $k_5$ & $k_6$ \\ 
\hline
\endfirsthead
\multicolumn{14}{c}
{\tablename\ \thetable\ -- \textit{Continued from previous page}} \\
\hline
\hline 
$l_1$ & $l_2$ & $l_3$ & $l_6$ & $g_1$ & $g_3$ & $h_{1}$ & $h_3$ & $k_1$ & $k_2$ & $k_3$ & $k_4$ & $k_5$ & $k_6$ \\ 
\hline
\endhead
\hline \multicolumn{14}{r}{\textit{Continued on next page}} \\
\endfoot
\hline
\endlastfoot
$5$ & $ 1$ & $ 0$ & $ 0$ & $ \frac{ -2}{3}$ & $ 0$ & $ \frac{ 32}{3}$ & $ 0$ & $ 0$ & $ 0$ & $ 0$ & $ 0$ & $ 0$ & $ 0$ \\
 $4$ & $ 0$ & $ 1$ & $ 1$ & $ \frac{ -8}{39}$ & $ \frac{ -6}{13}$ & $ \frac{ 128}{39}$ & $ \frac{ 96}{13}$ & $ \frac{ 64}{13}$ & $ \frac{ 272}{13}$ & $ \frac{ 160}{13}$ & $ \frac{ 576}{13}$ & $ 0$ & $ \frac{ 64}{13}$ \\
 $3$ & $ 3$ & $ 0$ & $ 0$ & $ \frac{ -2}{3}$ & $  0$ & $ \frac{ 16}{3}$ & $  0$ & $ \frac{ 4}{3}$ & $ 0$ & $ -4$ & $ \frac{ -16}{3}$ & $ \frac{ 16}{3}$ & $ \frac{ 16}{3}$ \\
 $3$ & $ 1$ & $ 2$ & $ 0$ & $ \frac{ 10}{39}$ & $ \frac{ -12}{13}$ & $ \frac{ 160}{39}$ & $ \frac{ -192}{13}$ & $ \frac{ 64}{39}$ & $ \frac{ 8}{13}$ & $ \frac{ -64}{13}$ & $ \frac{ -928}{39}$ & $ \frac{ -320}{39}$ & $ \frac{ -224}{39}$ \\
 $3$ & $ 1$ & $ 0$ & $ 2$ & $ \frac{ 10}{39}$ & $ \frac{ -12}{13}$ & $ \frac{ 80}{39}$ & $ \frac{ -96}{13}$ & $ \frac{ 48}{13}$ & $ \frac{ 168}{13}$ & $ \frac{ 188}{13}$ & $ \frac{ 576}{13}$ & $ \frac{ -80}{13}$ & $ \frac{ 80}{13}$ \\
 $2$ & $ 2$ & $ 1$ & $ 1$ & $ \frac{ -8}{39}$ & $ \frac{ -6}{13}$ & $ \frac{ 64}{39}$ & $ \frac{ 48}{13}$ & $ \frac{ 100}{39}$ & $ \frac{ 88}{13}$ & $ \frac{ 60}{13}$ & $ \frac{ 920}{39}$ & $ \frac{ 64}{39}$ & $ \frac{ 64}{39}$ \\
 $2$ & $ 0$ & $ 3$ & $ 1$ & $ \frac{ 4}{39}$ & $ \frac{ -10}{13}$ & $ \frac{ 64}{39}$ & $ \frac{ -160}{13}$ & $ \frac{ 88}{39}$ & $ \frac{ 24}{13}$ & $ \frac{ 16}{13}$ & $ \frac{ 128}{39}$ & $ \frac{ -128}{39}$ & $ \frac{ 160}{39}$ \\
 $2$ & $ 0$ & $ 1$ & $ 3$ & $ \frac{ 4}{39}$ & $ \frac{ -10}{13}$ & $ \frac{ 32}{39}$ & $ \frac{ -80}{13}$ & $ \frac{ 40}{13}$ & $ \frac{ 88}{13}$ & $ \frac{ 44}{13}$ & $ \frac{ 168}{13}$ & $ \frac{ -32}{13}$ & $ \frac{ 32}{13}$ \\
 $1$ & $ 5$ & $ 0$ & $ 0$ & $ \frac{ -2}{3}$ & $ 0$ & $ \frac{ 8}{3}$ & $  0$ & $ 0$ & $ 0$ & $ 0$ & $ 0$ & $ 0$ & $ 0$ \\
 $1$ & $ 3$ & $ 2$ & $ 0$ & $ \frac{ 10}{39}$ & $ \frac{ -12}{13}$ & $ \frac{ 80}{39}$ & $ \frac{ -96}{13}$ & $ \frac{ -4}{13}$ & $ \frac{ -40}{13}$ & $ \frac{ -20}{13}$ & $ \frac{ -256}{13}$ & $ \frac{ -80}{13}$ & $ \frac{ 80}{13}$ \\
 $1$ & $ 3$ & $ 0$ & $ 2$ & $ \frac{ 10}{39}$ & $ \frac{ -12}{13}$ & $ \frac{ 40}{39}$ & $ \frac{ -48}{13}$ & $ \frac{ 28}{39}$ & $ \frac{ 40}{13}$ & $ \frac{ 80}{13}$ & $ \frac{ 560}{39}$ & $ \frac{ 112}{39}$ & $ \frac{ 160}{39}$ \\
 $1$ & $ 1$ & $ 4$ & $ 0$ & $ \frac{ -2}{39}$ & $ \frac{ -8}{13}$ & $ \frac{ 32}{39}$ & $ \frac{ 128}{13}$ & $ \frac{ 16}{13}$ & $ \frac{ 16}{13}$ & $ \frac{ -64}{13}$ & $ \frac{ -64}{13}$ & $ 0$ & $ \frac{ -192}{13}$ \\
 $1$ & $ 1$ & $ 2$ & $ 2$ & $ \frac{ -2}{39}$ & $ \frac{ -8}{13}$ & $ \frac{ 16}{39}$ & $ \frac{ 64}{13}$ & $ \frac{ 64}{39}$ & $ \frac{ 48}{13}$ & $ \frac{ 28}{13}$ & $ \frac{ 464}{39}$ & $ \frac{ 16}{39}$ & $ \frac{ 16}{39}$ \\
 $1$ & $ 1$ & $ 0$ & $ 4$ & $ \frac{ -2}{39}$ & $ \frac{ -8}{13}$ & $ \frac{ 8}{39}$ & $ \frac{ 32}{13}$ & $ \frac{ 24}{13}$ & $ \frac{ 64}{13}$ & $ \frac{ 48}{13}$ & $ \frac{ 160}{13}$ & $ \frac{ -96}{13}$ & $ 0$ \\
 $0$ & $ 4$ & $ 1$ & $ 1$ & $ \frac{ -8}{39}$ & $ \frac{ -6}{13}$ & $ \frac{ 32}{39}$ & $ \frac{ 24}{13}$ & $ \frac{ -8}{13}$ & $ \frac{ 48}{13}$ & $ \frac{ -16}{13}$ & $ \frac{ 224}{13}$ & $ \frac{ 32}{13}$ & $ 0$ \\
 $0$ & $ 2$ & $ 3$ & $ 1$ & $ \frac{ 4}{39}$ & $ \frac{ -10}{13}$ & $ \frac{ 32}{39}$ & $ \frac{ -80}{13}$ & $ \frac{ -12}{13}$ & $ \frac{ -16}{13}$ & $ \frac{ 44}{13}$ & $ \frac{ -40}{13}$ & $ \frac{ -32}{13}$ & $ \frac{ 32}{13}$ \\
 $0$ & $ 2$ & $ 1$ & $ 3$ & $ \frac{ 4}{39}$ & $ \frac{ -10}{13}$ & $ \frac{ 16}{39}$ & $ \frac{ -40}{13}$ & $ \frac{ -20}{39}$ & $ \frac{ 16}{13}$ & $ \frac{ 32}{13}$ & $ \frac{ 224}{39}$ & $ \frac{ -80}{39}$ & $ \frac{ 64}{39}$ \\
 $0$ & $ 0$ & $ 5$ & $ 1$ & $ 0$ & $ \frac{ -2}{3}$ & $ 0$ & $ \frac{ 32}{3}$ & $ 0$ & $ 0$ & $ 0$ & $ 0$ & $ 0$ & $ 0$ \\
 $0$ & $ 0$ & $ 3$ & $ 3$ & $ 0$ & $ \frac{ -2}{3}$ & $ 0$ & $ \frac{ 16}{3}$ & $ 0$ & $ 0$ & $ \frac{ 4}{3}$ & $ \frac{ 8}{3}$ & $ 0$ & $ 0$ \\
 $0$ & $ 0$ & $ 1$ & $ 5$ & $ 0$ & $ \frac{ -2}{3}$ & $ 0$ & $ \frac{ 8}{3}$ & $ 0$ & $ 0$ & $ 0$ & $ 0$ & $ 0$ & $ 0$
\end{longtable}
\end{center} }
%\end{landscape}

%\begin{landscape}
{\small \begin{center}
\renewcommand*{\arraystretch}{1.3} \begin{longtable}{r r r r | c c c c c c c c c c }
\caption{Values of $l_d$, $m_1$, $m_2$, $m_{3}$, $m_4$, $n_j$ for Theorems \ref{th:9_2} and \ref{th:9_3} }\label{table:9_4} \\
\hline\hline 
$l_1$ & $l_2$ & $l_3$ & $l_6$ & $m_1$ & $m_2$ & $m_{3}$ & $m_4$ & $n_1$ & $n_2$ & $n_3$ & $n_4$ & $n_5$ & $n_6$ \\ 
\hline
\endfirsthead
\multicolumn{14}{c}
{\tablename\ \thetable\ -- \textit{Continued from previous page}} \\
\hline
\hline 
$l_1$ & $l_2$ & $l_3$ & $l_6$ & $m_1$ & $m_2$ & $m_{3}$ & $m_4$ & $n_1$ & $n_2$ & $n_3$ & $n_4$ & $n_5$ & $n_6$\\ 
\hline
\endhead
\hline \multicolumn{14}{r}{\textit{Continued on next page}} \\
\endfoot
\hline
\endlastfoot
$5$ & $ 0$ & $ 0$ & $ 1$ & $ \frac{ -1}{23}$ & $ \frac{ 144}{23}$ & $ \frac{ 16}{23}$ & $ \frac{ -9}{23}$ & $ \frac{ 80}{23}$ & $ \frac{ 480}{23}$ & $ \frac{ 1600}{23}$ & $ 0$ & $ \frac{ 320}{23}$ & $ 0$ \\
 $4$ & $ 1$ & $ 1$ & $ 0$ & $ \frac{ -1}{23}$ & $ \frac{ 144}{23}$ & $ \frac{ -16}{23}$ & $ \frac{ 9}{23}$ & $ \frac{ 48}{23}$ & $ \frac{ 16}{23}$ & $ \frac{ 48}{23}$ & $ \frac{ -112}{23}$ & $ \frac{ -384}{23}$ & $ \frac{ 256}{23}$ \\
 $3$ & $ 2$ & $ 0$ & $ 1$ & $ \frac{ -1}{23}$ & $ \frac{ 72}{23}$ & $ \frac{ 8}{23}$ & $ \frac{ -9}{23}$ & $ \frac{ 68}{23}$ & $ \frac{ 172}{23}$ & $ \frac{ 496}{23}$ & $ \frac{ 80}{23}$ & $ \frac{ 272}{23}$ & $ \frac{ -80}{23}$ \\
 $3$ & $ 0$ & $ 2$ & $ 1$ & $ \frac{ -1}{23}$ & $ \frac{ 48}{23}$ & $ \frac{ -16}{23}$ & $ \frac{ 3}{23}$ & $ \frac{ 104}{23}$ & $ \frac{ 128}{23}$ & $ \frac{ 320}{23}$ & $ \frac{ 32}{23}$ & $ \frac{ 32}{23}$ & $ \frac{ 64}{23}$ \\
 $3$ & $ 0$ & $ 0$ & $ 3$ & $ \frac{ -1}{23}$ & $ \frac{ 24}{23}$ & $ \frac{ -8}{23}$ & $ \frac{ 3}{23}$ & $ \frac{ 120}{23}$ & $ \frac{ 272}{23}$ & $ \frac{ 772}{23}$ & $ \frac{ -132}{23}$ & $ \frac{ -16}{23}$ & $ \frac{ 48}{23}$ \\
 $2$ & $ 3$ & $ 1$ & $ 0$ & $ \frac{ -1}{23}$ & $ \frac{ 72}{23}$ & $ \frac{ -8}{23}$ & $ \frac{ 9}{23}$ & $ \frac{ 20}{23}$ & $ \frac{ -60}{23}$ & $ \frac{ -56}{23}$ & $ \frac{ 96}{23}$ & $ \frac{ -112}{23}$ & $ \frac{ 192}{23}$ \\
 $2$ & $ 1$ & $ 3$ & $ 0$ & $ \frac{ -1}{23}$ & $ \frac{ 48}{23}$ & $ \frac{ 16}{23}$ & $ \frac{ -3}{23}$ & $ \frac{ 32}{23}$ & $ \frac{ 40}{23}$ & $ \frac{ 24}{23}$ & $ \frac{ -312}{23}$ & $ \frac{ -288}{23}$ & $ 0$ \\
 $2$ & $ 1$ & $ 1$ & $ 2$ & $ \frac{ -1}{23}$ & $ \frac{ 24}{23}$ & $ \frac{ 8}{23}$ & $ \frac{ -3}{23}$ & $ \frac{ 64}{23}$ & $ \frac{ 136}{23}$ & $ \frac{ 396}{23}$ & $ \frac{ -52}{23}$ & $ \frac{ 80}{23}$ & $ \frac{ -32}{23}$ \\
 $1$ & $ 4$ & $ 0$ & $ 1$ & $ \frac{ -1}{23}$ & $ \frac{ 36}{23}$ & $ \frac{ 4}{23}$ & $ \frac{ -9}{23}$ & $ \frac{ 16}{23}$ & $ \frac{ 64}{23}$ & $ \frac{ 128}{23}$ & $ \frac{ -64}{23}$ & $ \frac{ 64}{23}$ & $ \frac{ 64}{23}$ \\
 $1$ & $ 2$ & $ 2$ & $ 1$ & $ \frac{ -1}{23}$ & $ \frac{ 24}{23}$ & $ \frac{ -8}{23}$ & $ \frac{ 3}{23}$ & $ \frac{ 28}{23}$ & $ \frac{ -4}{23}$ & $ \frac{ 128}{23}$ & $ \frac{ 144}{23}$ & $ \frac{ -16}{23}$ & $ \frac{ 48}{23}$ \\
 $1$ & $ 2$ & $ 0$ & $ 3$ & $ \frac{ -1}{23}$ & $ \frac{ 12}{23}$ & $ \frac{ -4}{23}$ & $ \frac{ 3}{23}$ & $ \frac{ 36}{23}$ & $ \frac{ 68}{23}$ & $ \frac{ 308}{23}$ & $ \frac{ 108}{23}$ & $ \frac{ 144}{23}$ & $ \frac{ -144}{23}$ \\
 $1$ & $  0$ & $ 4$ & $ 1$ & $ \frac{ -1}{23}$ & $ \frac{ 16}{23}$ & $ \frac{ 16}{23}$ & $ \frac{ -1}{23}$ & $ \frac{ 16}{23}$ & $ \frac{ 16}{23}$ & $ \frac{ 112}{23}$ & $ \frac{ -48}{23}$ & $ 0$ & $ 0$ \\
 $1$ & $ 0$ & $ 2$ & $ 3$ & $ \frac{ -1}{23}$ & $ \frac{ 8}{23}$ & $ \frac{ 8}{23}$ & $ \frac{ -1}{23}$ & $ \frac{ 32}{23}$ & $ \frac{ 32}{23}$ & $ \frac{ 148}{23}$ & $ \frac{ -4}{23}$ & $ \frac{ 16}{23}$ & $ \frac{ -16}{23}$ \\
 $1$ & $ 0$ & $ 0$ & $ 5$ & $ \frac{ -1}{23}$ & $ \frac{ 4}{23}$ & $ \frac{ 4}{23}$ & $ \frac{ -1}{23}$ & $ \frac{ 40}{23}$ & $ \frac{ 40}{23}$ & $ \frac{ 120}{23}$ & $ \frac{ -120}{23}$ & $ \frac{ -160}{23}$ & $ \frac{ 160}{23}$ \\
 $0$ & $ 5$ & $ 1$ & $ 0$ & $ \frac{ -1}{23}$ & $ \frac{ 36}{23}$ & $ \frac{ -4}{23}$ & $ \frac{ 9}{23}$ & $ \frac{ -40}{23}$ & $ \frac{ 40}{23}$ & $ \frac{ -200}{23}$ & $ \frac{ 200}{23}$ & $ \frac{ -160}{23}$ & $ \frac{ 160}{23}$ \\
 $0$ & $ 3$ & $ 3$ & $ 0$ & $ \frac{ -1}{23}$ & $ \frac{ 24}{23}$ & $ \frac{ 8}{23}$ & $ \frac{ -3}{23}$ & $ \frac{ -28}{23}$ & $ \frac{ 44}{23}$ & $ \frac{ -64}{23}$ & $ \frac{ -328}{23}$ & $ \frac{ 80}{23}$ & $ \frac{ -32}{23}$ \\
 $0$ & $ 3$ & $ 1$ & $ 2$ & $ \frac{ -1}{23}$ & $ \frac{ 12}{23}$ & $ \frac{ 4}{23}$ & $ \frac{ -3}{23}$ & $ \frac{ -12}{23}$ & $ 4$ & $ \frac{ 76}{23}$ & $ \frac{ -60}{23}$ & $ \frac{ 80}{23}$ & $ \frac{ -48}{23}$ \\
 $0$ & $ 1$ & $ 5$ & $ 0$ & $ \frac{ -1}{23}$ & $ \frac{ 16}{23}$ & $ \frac{ -16}{23}$ & $ \frac{ 1}{23}$ & $ 0$ & $ \frac{ -80}{23}$ & $ \frac{ -80}{23}$ & $ \frac{ 80}{23}$ & $ \frac{ -320}{23}$ & $ 0$ \\
 $0$ & $ 1$ & $ 3$ & $ 2$ & $ \frac{ -1}{23}$ & $ \frac{ 8}{23}$ & $ \frac{ -8}{23}$ & $ \frac{ 1}{23}$ & $ 0$ & $ \frac{ -16}{23}$ & $ \frac{ 36}{23}$ & $ \frac{ 68}{23}$ & $ \frac{ 16}{23}$ & $ 0$ \\
 $0$ & $ 1$ & $ 1$ & $ 4$ & $ \frac{ -1}{23}$ & $ \frac{ 4}{23}$ & $ \frac{ -4}{23}$ & $ \frac{ 1}{23}$ & $ 0$ & $ \frac{ 16}{23}$ & $ \frac{ 48}{23}$ & $ \frac{ 16}{23}$ & $ 0$ & $ 0$
\end{longtable}
\end{center} }
%\end{landscape}

\section{Newforms in $M_3(\Gamma_0(24),\chi)$}\label{sec:5}
In this section we give another use of the bases provided in Section \ref{sec:8_1}. In \cite{martinono} Martin and Ono expressed all weight $2$ newforms that are eta quotients. Below we express the newforms in $M_3(\Gamma_0(24),\chi)$ in terms of eta quotients from Section \ref{sec:8_1}.
\begin{thm} Let
\beqars
& f_1(z):=& S(3,24,\chi_{_{-3}})_1 +3 S(3,24,\chi_{_{-3}})_3 +4 S(3,24,\chi_{_{-3}})_4\\
&&+\alpha_1 S(3,24,\chi_{_{-3}})_3 ,\\
& f_2(z):=& S(3,24,\chi_{_{-8}})_1+ 2S(3,24,\chi_{_{-8}})_2+2S(3,24,\chi_{_{-8}})_4 -2S(3,24,\chi_{_{-8}})_5 \\
&& \quad \qquad-4 S(3,24,\chi_{_{-8}})_6 \\
&& +\alpha_2 ( S(3,24,\chi_{_{-8}})_2+ 3/2S(3,24,\chi_{_{-8}})_3+ 5S(3,24,\chi_{_{-8}})_4 \\
&& \quad \qquad - S(3,24,\chi_{_{-8}})_5 + 3S(3,24,\chi_{_{-8}})_6 )\\
&& +\alpha^2_2 ( -1/2 S(3,24,\chi_{_{-8}})_3 - S(3,24,\chi_{_{-8}})_6 ) \\
&& +\alpha^3_2 ( 1/4 S(3,24,\chi_{_{-8}})_3+ 1/2S(3,24,\chi_{_{-8}})_4 -1/2S(3,24,\chi_{_{-8}})_5 \\
&& \quad \qquad + 1/2S(3,24,\chi_{_{-8}})_6 ), \\
& f_3(z):= & S(3,24,\chi_{_{-24}})_1 - S(3,24,\chi_{_{-24}})_2 + 3 S(3,24,\chi_{_{-24}})_3 + 7 S(3,24,\chi_{_{-24}})_4 \\
&& \quad \qquad + 8 S(3,24,\chi_{_{-24}})_5 - 4 S(3,24,\chi_{_{-24}})_6 , \\
& f_4(z):= & S(3,24,\chi_{_{-24}})_1 + 3 S(3,24,\chi_{_{-24}})_2 + 5 S(3,24,\chi_{_{-24}})_3 \\
&& \quad \qquad + S(3,24,\chi_{_{-24}})_4 - 4 S(3,24,\chi_{_{-24}})_6, \\
& f_5(z):= & S(3,24,\chi_{_{-24}})_1 + S(3,24,\chi_{_{-24}})_2 + S(3,24,\chi_{_{-24}})_3 \\
&& \quad \qquad -3 S(3,24,\chi_{_{-24}})_4 -6 S(3,24,\chi_{_{-24}})_5 + 6 S(3,24,\chi_{_{-24}})_6\\
&& +\alpha_3 \left( S(3,24,\chi_{_{-24}})_2 + 3/2 S(3,24,\chi_{_{-24}})_3 -1/2 S(3,24,\chi_{_{-24}})_4 \right. \\
&& \quad \qquad \left. + 3 S(3,24,\chi_{_{-24}})_5 \right)\\
&& +\alpha^2_3 ( - S(3,24,\chi_{_{-24}})_3 - S(3,24,\chi_{_{-24}})_5 + S(3,24,\chi_{_{-24}})_6 )\\
&& +\alpha^3_3 ( -1/4 S(3,24,\chi_{_{-24}})_3 -1/4 S(3,24,\chi_{_{-24}})_4 + 1/2S(3,24,\chi_{_{-24}})_5 ),
\eeqars
where $\alpha_1^2-2\alpha_1+9=0$, $\alpha_2^4  - 2 \alpha_2^3 + 6 \alpha_2^2 - 8 \alpha_2 + 16=0$ and $\alpha_3^4  + 6 \alpha_3^2 +16=0$. Then we have
\begin{align*}
 f_1(z)\in S^{new}_3(\Gamma_0(24),\chi_{_{-3}})\\
 f_2(z) \in S^{new}_3(\Gamma_0(24),\chi_{_{-8}})\\
 f_3(z), f_4(z), f_5(z) \in S^{new}_3(\Gamma_0(24),\chi_{_{-24}}).
 \end{align*}
\end{thm}
\begin{proof}
First $10$ Fourier coefficients of newforms in $S^{new}_3(\Gamma_0(24),\chi)$ are given in \cite{lmfdb}, for example, we have
\begin{align}
F(z):=q + \alpha_1 q^3 + (-2 \alpha_1 +2)q^5 - 6q^7 + (2\alpha_1 -9)q^9 +O(q^{10}) \in S^{new}_3(\Gamma_0(24),\chi_{_{-3}}). \label{eq:3}
\end{align}
On the other hand, $F(z) \in S_3(\Gamma_0(24),\chi)$, that is by Theorem \ref{th:8_1} we have
\begin{align}
F(z):=x_1 S(3,24,\chi_{_{-3}})_1+x_2 S(3,24,\chi_{_{-3}})_2+x_3 S(3,24,\chi_{_{-3}})_3+x_4 S(3,24,\chi_{_{-3}})_4, \label{eq:2}
\end{align}
for some $x_1,x_2, x_3,x_4 \in \cc$. We expand the eta quotients in (\ref{eq:2}), and compare first ten coefficients of (\ref{eq:2}) with (\ref{eq:3}). We solve the resulting linear equations and find 
\beqars
x_1=1, ~x_2=0, ~x_3=\alpha_1+3, ~x_4= 4. 
\eeqars
The rest can be proven similarly.
\end{proof}

\section*{Remarks}
We did some further investigations on the spaces mentioned in this paper. Using methods similar to \cite{alacaaygin}, we determined that there are $6332$, $6288$, $2424$ and $2424$ eta quotients in $M_3(\Gamma_0(24),\chi_{_{-3}})$, $M_3(\Gamma_0(24),\chi_{_{-4}})$, $M_3(\Gamma_0(24),\chi_{_{-8}})$ and $M_3(\Gamma_0(24),\chi_{_{-24}})$, respectively. Using methods similar to \cite{alacaaygin5} we find, among them $140$, $40$, $4$ and none can be written in terms of Eisenstein series, respectively. This allows us to determine the Fourier coefficients of these eta quotients in terms of sum of divisors functions defined by (\ref{3_1}). Below, as an example, we list the eta quotients which can be written with no more than two Eisenstein series.

\begin{align*}
 \eta_{3}[-3, 9](z)&= \sum_{n=1}^\infty \sigma_{3,\chi_{_{1}},\chi_{_{-3}}}(n) q^n, \\
 \eta_{6}[-4, 5, 4, 1](z)&= \sum_{n=1}^\infty (\sigma_{3,\chi_{_{1}},\chi_{_{-3}}}(n)+ \sigma_{3,\chi_{_{1}},\chi_{_{-3}}}(n/2)) q^n, \\ 
 \eta_{6}[4, 1, -4, 5](z)&= \sum_{n=1}^\infty (\sigma_{3,\chi_{_{-3}},\chi_{_{1}}}(n) - \sigma_{3,\chi_{_{-3}},\chi_{_{1}}}(n/2)) q^n,\\
 \eta_{4}[-4, 6, 4](z)&=\sum_{n=1}^\infty \sigma_{3,\chi_{_{1}},\chi_{_{-4}}}(n) q^n, \\
 \eta_{8}[-4, 2, 16, -8](z)&= 1+ 4 \sum_{n=1}^\infty ( \sigma_{3,\chi_{_{1}},\chi_{_{-4}}}(n) - \sigma_{3,\chi_{_{-4}},\chi_{_{1}}}(n/2) ) q^n, \\
 \eta_{4}[-12, 30, -12](z)&= 1+4 \sum_{n=1}^{\infty} (4 \sigma_{3,\chi_{_{1}},\chi_{_{-4}}}(n) - \sigma_{3,\chi_{_{-4}},\chi_{_{1}}}(n) ) q^n, \\
 \eta_{4}[4, -6, 8](z)&= \sum_{n=1}^{\infty} ( \sigma_{3,\chi_{_{1}},\chi_{_{-4}}}(n) - 8 \sigma_{3,\chi_{_1},\chi_{_{-4}}}(n/2) ) q^n,\\
 \eta_{4}[-4, 18, -8](z)&= 1+4 \sum_{n=1}^{\infty} ( \sigma_{3,\chi_{_{-4}},\chi_{_{1}}}(n) - 2 \sigma_{3,\chi_{_{-4}},\chi_{_{1}}}(n/2) ) q^n, \\
 \eta_{8}[-2, -5, 23, -10](z)&= 1+ \frac{2}{3}\sum_{n=1}^{\infty} ( 4 \sigma_{3,\chi_{_{1}},\chi_{_{-8}}}(n) - \sigma_{3,\chi_{_{-8}},\chi_{_{1}}}(n) ) q^n.
\end{align*}
Some of these equations were previously known.

\section*{Acknowledgments} 
The author was supported by the Singapore Ministry of Education Academic Research Fund, Tier 2, project number MOE2014-T2-1-051, ARC40/14.

\end{document}